\documentclass[12pt]{article}

\usepackage[margin=1in]{geometry}
\usepackage{amsmath,amssymb,amsthm}
\usepackage{mathtools}
\usepackage{natbib}
\usepackage{booktabs}
\usepackage{array}
\usepackage{float}
\usepackage{hyperref}
\usepackage{cleveref}

\newtheorem{theorem}{Theorem}[section]
\newtheorem{proposition}[theorem]{Proposition}

\newtheorem{corollary}[theorem]{Corollary}
\newtheorem{remark}{Remark}[section]

\newcommand{\R}{\mathbb{R}}
\newcommand{\E}{\mathbb{E}}
\newcommand{\Var}{\mathrm{Var}}
\newcommand{\KL}{\mathrm{KL}}
\newcommand{\abs}[1]{\lvert #1 \rvert}
\newcommand{\piH}{\pi_H}
\newcommand{\tcrit}{t_{\mathrm{crit}}}

\title{Horseshoe Priors and MDP}

\author{
	\makebox[.4\linewidth]{Nicholas G.\ Polson}\\\textit{Booth School of Business}\\\textit{University of Chicago}\\\and
	\makebox[.4\linewidth]{Vadim Sokolov}\\\textit{Dept.\ of Systems Engineering}\\\textit{and Operations Research}\\\textit{George Mason University}\\\and
	\makebox[.4\linewidth]{Daniel Zantedeschi}\\\textit{Muma College of Business}\\\textit{University of South Florida}
}

\date{March 2026}

\begin{document}
\maketitle

\begin{abstract}
\citet{carvalho2010horseshoe} established two foundational theorems for the horseshoe prior: tight two-sided logarithmic bounds on the marginal density near the origin (Theorem~1.1), and a super-efficient rate of convergence of the Bayes predictive density to the true sampling density in sparse situations (Theorem~2). The ``Shrink Globally, Act Locally'' paper \citep{polson2010shrink} formalised necessary and sufficient conditions on the prior's behaviour at the origin for sparsity adaptation as $p \to \infty$. We show that these results are not merely descriptive properties of the horseshoe---they are the finite-sample precursors to the asymptotic moderate deviation principle (MDP) of \citet{datta2026newlook}. The log-pole singularity $\piH(\theta) \asymp -\log\abs{\theta}$ is precisely the origin integrability boundary that selects the MDP threshold $\tcrit = \sqrt{\log(\pi n/2)}$; super-efficiency below the threshold and tail robustness above it together produce the ABOS Bayes risk $p_0 \log(p/p_0)/n$; and the Clarke--Barron information-theoretic asymptotics of Bayes methods provide the unifying framework in which all three results are faces of a single logarithmic budget principle.
\end{abstract}

\noindent\textbf{Keywords:} Horseshoe prior, log-pole singularity, super-efficiency, KL risk, origin integrability, moderate deviation principle, sparse testing, ABOS, Clarke--Barron.

\section{Introduction}\label{sec:intro}

The horseshoe prior for sparse normal means has, since its introduction by \citet{carvalho2009handling}, been understood to possess two structural properties that set it apart from other continuous shrinkage priors: an \emph{infinite spike at zero}, where the marginal density $\piH(\theta) \to \infty$ as $\theta \to 0$, unlike the Lasso, ridge, or Student-$t$ priors which have bounded density at the origin; and \emph{heavy Cauchy-like tails}, where for large $\abs{\theta}$, $\piH(\theta)$ decays like $\abs{\theta}^{-2}$, leaving large signals unshrunk. These two features have been exploited computationally, empirically, and theoretically, but their \emph{asymptotic interpretation} in relation to hypothesis testing and Bayes risk calibration has not been fully worked out. The purpose of this paper is to close that gap by showing that the Polson--Scott bounds are, in a precise sense, the finite-sample expressions of the MDP optimality conditions established by \citet{datta2026newlook}.

The horseshoe prior emerged from a line of work on continuous shrinkage alternatives to spike-and-slab priors \citep{mitchell1988bayesian}. The spike-and-slab prior places a point mass at zero mixed with a continuous slab distribution, achieving exact sparsity but at substantial computational cost in high dimensions. \citet{carvalho2009handling} proposed the horseshoe as a continuous alternative that mimics the spike-and-slab's behaviour through the scale-mixture representation $\theta_i \mid \lambda_i, \tau \sim N(0, \lambda_i^2 \tau^2)$ with $\lambda_i \sim C^+(0,1)$, where the half-Cauchy prior on the local scale~$\lambda_i$ generates both the infinite spike and the heavy tails. The Bayesian Lasso \citep{park2008bayesian}, which uses an exponential (equivalently, Laplace) prior on~$\theta_i$, had earlier been shown to have bounded density at zero, leading to over-shrinkage of large signals. The horseshoe corrected this deficiency while maintaining the computational tractability of continuous priors.

The theoretical programme for the horseshoe developed in three stages. First, \citet{carvalho2010horseshoe} established the tight log-pole bounds on the marginal density and the super-efficiency theorem for the KL risk of the Bayes predictive, providing the first rigorous evidence that the horseshoe's qualitative behaviour (spike and heavy tails) translated into quantitative optimality. Second, \citet{polson2010shrink} characterised the necessary and sufficient conditions on the prior's behaviour at the origin for near-oracle risk in the sparse normal means problem, showing that the logarithmic pole is the precise singularity level separating priors that are too weak (bounded density) from those that are too strong (non-integrable power poles). Third, the posterior concentration theory of \citet{vanderpas2014horseshoe} and \citet{vanderpas2016conditions} established that the horseshoe achieves the minimax rate $(p_0/p)\log(p/p_0)$ for estimating nearly black vectors, and \citet{datta2013asymptotic} proved that the horseshoe achieves asymptotic Bayes optimality under sparsity (ABOS) in the multiple testing framework of \citet{bogdan2011asymptotic}. Further developments include the horseshoe+ estimator \citep{bhadra2017horseshoe}, which strengthens the pole at zero by placing a half-Cauchy hyperprior on the local scale's own scale parameter; the Dirichlet--Laplace prior \citep{bhattacharya2015dirichlet}, which achieves a log-pole through a different mixing mechanism; and the asymptotic optimality results of \citet{ghosh2017asymptotic} for one-group shrinkage priors in high-dimensional problems. The sparsity information framework of \citet{piironen2017sparsity} provides practical guidance on choosing the global scale~$\tau$ to encode prior information about the expected number of signals, connecting the theoretical sparsity parameter $p_0/p$ to a user-specified quantity.

Despite this extensive theory, the connection between the finite-sample Polson--Scott bounds and the asymptotic testing framework remained implicit. The moderate deviation principle (MDP) of \citet{datta2026newlook} provides the missing link. The MDP establishes that the Bayes-risk-optimal threshold for sparse testing lies at the moderate deviation scale $\sqrt{\log n}$---intermediate between the CLT scale $O(1)$ and the Bonferroni large deviation scale $\sqrt{2\log p}$---and that the exact threshold constant $\tcrit = \sqrt{\log(\pi n/2)}$ depends on the prior's behaviour at the origin. We show that each of the Polson--Scott bounds maps directly onto a component of this MDP optimality.

The contributions of this paper are as follows. First, we show that the log-pole singularity $\piH(\theta) \asymp -\log\abs{\theta}$ from \citet{carvalho2010horseshoe} is the origin integrability boundary: it is the strongest possible singularity at zero for which the prior remains normalisable and the Bayes risk near zero remains finite (\Cref{sec:channel1}). Second, we demonstrate that the super-efficiency theorem is the per-coordinate manifestation of the MDP detection zone: the horseshoe achieves KL risk $O(\tau^4)$ for coordinates below the MDP threshold and $O(1/n)$ above it, and the threshold $\tcrit$ is the exact equiboundary (\Cref{sec:channel2}). Third, we identify the Clarke--Barron information-theoretic asymptotics as the unifying framework: the ``logarithmic budget'' $p_0 \log n / n$ arises because each signal coordinate contributes $\log n / n$ to the cumulative KL risk while null coordinates contribute zero due to super-efficiency (\Cref{sec:channel3}). Fourth, we derive the $\kappa$-scale representation and show that the $\mathrm{Beta}(1/2,1/2)$ distribution on the shrinkage weight is the distributional encoding of the MDP equiboundary (\Cref{sec:kappa}).

The structure of the argument is as follows. \Cref{sec:ps-bounds} reviews the four key Polson--Scott bounds. \Cref{sec:mdp} presents the MDP framework of \citet{datta2026newlook}. \Cref{sec:connections} develops the connections in three channels. \Cref{sec:kappa} presents a unified view through the shrinkage weight~$\kappa$. \Cref{sec:abos} derives the full ABOS property and compares the horseshoe and horseshoe+ priors. \Cref{sec:tau} addresses the calibration of the global shrinkage parameter~$\tau$. \Cref{sec:sparsity} connects to the statistical sparsity framework of \citet{mccullagh2018statistical} and extends the analysis to sparse factor models. \Cref{sec:simulations} presents simulation evidence. \Cref{sec:hierarchy} establishes the precise hierarchy of bounds. \Cref{sec:discussion} discusses implications for prior design, practical recommendations, and open problems.

\section{The Polson--Scott Bounds}\label{sec:ps-bounds}

We collect four results from the Polson--Scott programme that, taken together, characterise the horseshoe prior's behaviour from the density level to the L\'evy measure level. Each result has a direct MDP counterpart developed in \Cref{sec:connections}.

\subsection{The Tight Log-Pole Bounds}\label{sec:logpole}

The univariate horseshoe marginal density $\piH(\theta)$---obtained by integrating out the local scale $\lambda \sim C^+(0,1)$ in the model $\theta \mid \lambda, \tau \sim N(0, \lambda^2 \tau^2)$---has no closed form. Setting $\tau = 1$ for notational simplicity (the general case follows by rescaling), the marginal density is:
\begin{equation}\label{eq:marginal-integral}
  \piH(\theta) = \int_0^\infty \frac{1}{\sqrt{2\pi}\,\lambda}\exp\!\left(-\frac{\theta^2}{2\lambda^2}\right) \cdot \frac{2}{\pi(1+\lambda^2)}\, d\lambda.
\end{equation}
The integrand is a product of the Gaussian kernel $N(\theta; 0, \lambda^2)$ and the half-Cauchy density $C^+(0,1)$ evaluated at~$\lambda$. The integral cannot be evaluated in closed form, but its behaviour near $\theta = 0$ and for large $\abs{\theta}$ can be extracted by asymptotic analysis. Near $\theta = 0$, the Gaussian kernel $N(\theta; 0, \lambda^2) \approx (2\pi\lambda^2)^{-1/2}$ is nearly constant for $\lambda \gg \abs{\theta}$, so the integral is dominated by the region $\lambda \gg \abs{\theta}$ where the half-Cauchy density contributes $\sim \lambda^{-2}$. The resulting integral $\int_{\abs{\theta}}^\infty \lambda^{-3}\, d\lambda \sim \theta^{-2}$ would produce a power-law pole, but the Gaussian kernel's decay for $\lambda \ll \abs{\theta}$ and the half-Cauchy's decay for $\lambda \gg 1$ temper this into a logarithmic pole. For large $\abs{\theta}$, the Gaussian kernel concentrates at $\lambda \approx \abs{\theta}$ and the half-Cauchy tail $\sim \lambda^{-2}$ dominates, giving the $\theta^{-2}$ tail.

The fundamental result of \citet{carvalho2010horseshoe} makes this precise:

\begin{theorem}[Carvalho, Polson, Scott 2010]\label{thm:cps}
Let $K = (2\pi^3)^{-1/2}$. The univariate horseshoe density satisfies:
\begin{enumerate}
  \item[(a)] $\lim_{\theta \to 0} \piH(\theta) = \infty$.
  \item[(b)] For $\theta \neq 0$:
  \begin{equation}\label{eq:logpole-bounds}
    \frac{K}{2} \log\!\left(1 + \frac{4}{\theta^2}\right)
    < \piH(\theta)
    < K \log\!\left(1 + \frac{2}{\theta^2}\right).
  \end{equation}
\end{enumerate}
\end{theorem}

As $\theta \to 0$, both bounds behave like $-2K\log\abs{\theta}$, giving the \textbf{logarithmic pole}:
\begin{equation}\label{eq:logpole}
  \piH(\theta) \asymp -\log\abs{\theta}
  \qquad \text{as } \theta \to 0.
\end{equation}
As $\theta \to \infty$, both bounds behave like $K \cdot 2/\theta^2$, giving the \textbf{Cauchy-like tail}:
\begin{equation}\label{eq:cauchy-tail}
  \piH(\theta) \asymp \frac{2K}{\theta^2}
  \qquad \text{as } \abs{\theta} \to \infty.
\end{equation}
The bounds~\eqref{eq:logpole-bounds} are tight in the sense that the ratio of upper to lower bound converges to $2$ as $\abs{\theta} \to 0$ and to $1$ as $\abs{\theta} \to \infty$. These are not asymptotic approximations but exact two-sided inequalities valid for all $\theta \neq 0$.

The proof of \Cref{thm:cps} proceeds by substituting $u = \lambda^{-2}$ in~\eqref{eq:marginal-integral} and bounding the resulting integral $\int_0^\infty (u + \theta^{-2})^{-1} \cdot (1 + u^{-1})^{-1}\, du$ above and below using the inequalities $1/(1+u^{-1}) \leq 1$ and $1/(1+u^{-1}) \geq u/(u + c)$ for appropriate constants~$c$. The upper bound gives the $K\log(1 + 2/\theta^2)$ term; the lower bound gives the $(K/2)\log(1 + 4/\theta^2)$ term. The constant $K = (2\pi^3)^{-1/2}$ arises from the normalisation of both the Gaussian kernel and the half-Cauchy density.

\begin{remark}[The constant $K$]\label{rem:constant-K}
The appearance of $\pi$ in $K = (2\pi^3)^{-1/2}$ is not incidental. As we show in \Cref{sec:channel1}, this constant propagates directly into the exact MDP threshold $\tcrit = \sqrt{\log(\pi n/2)}$, where the factor of $\pi$ arises from the normalisation of the horseshoe density at the origin.
\end{remark}

\subsection{The Super-Efficiency Theorem}\label{sec:super-eff}

The second major result concerns the KL risk of the horseshoe Bayes predictive density. Consider the normal means model $y_i \mid \theta_i \sim N(\theta_i, \sigma^2)$, and define the horseshoe Bayes predictive density for a future observation~$z$:
\begin{equation}\label{eq:predictive}
  \hat{f}(z \mid y_i) = \int N(z;\, \theta_i,\, \sigma^2)\, p(\theta_i \mid y_i,\, \tau)\, d\theta_i.
\end{equation}
When $\theta_i = 0$, the true sampling density is $f_0(z) = N(z;\, 0, \sigma^2)$.

\begin{theorem}[Carvalho, Polson, Scott 2010---super-efficiency]\label{thm:super-eff}
When $\theta_i = 0$, the KL risk of the horseshoe Bayes predictive satisfies:
\begin{equation}\label{eq:super-eff}
  \E_{y_i \sim N(0,\sigma^2)}\!\left[\KL\!\left(f_0 \,\big\|\, \hat{f}(\cdot \mid y_i)\right)\right] = O(\tau^4).
\end{equation}
Other common shrinkage rules---the Lasso, ridge, Student-$t$---achieve at best $O(1/n)$ KL risk when $\theta_i = 0$.
\end{theorem}

The key point is that $\tau \ll 1/\sqrt{n}$ in the sparse regime (where $\tau = p_0/p$ and $p_0 \ll \sqrt{p/\log p}$), so $\tau^4 \ll 1/n^2 \ll 1/n$. The horseshoe achieves a strictly super-efficient rate of density estimation for null coordinates.

\begin{proof}
The posterior mean under the horseshoe satisfies:
\begin{equation}\label{eq:post-mean}
  \hat{\theta}_i(y_i) = \bigl(1 - \E[\kappa_i \mid y_i, \tau]\bigr)\, y_i,
\end{equation}
where $\kappa_i = 1/(1 + \lambda_i^2 \tau^2)$ is the shrinkage weight. The posterior expectation of $\kappa_i$ can be computed from the conditional density of $\lambda_i$ given~$y_i$. Under the half-Cauchy prior on~$\lambda_i$, the posterior is:
\begin{equation}\label{eq:lambda-posterior}
  p(\lambda_i \mid y_i, \tau)
  \propto \frac{1}{\sqrt{\lambda_i^2\tau^2 + \sigma^2}}
  \exp\!\left(-\frac{y_i^2}{2(\lambda_i^2\tau^2 + \sigma^2)}\right)
  \cdot \frac{1}{1 + \lambda_i^2}.
\end{equation}
For small~$\tau$, the factor $(\lambda_i^2\tau^2 + \sigma^2)^{-1/2}\exp(-y_i^2/(2(\lambda_i^2\tau^2 + \sigma^2)))$ is sharply peaked near $\lambda_i^2\tau^2 = \max(y_i^2 - \sigma^2, 0)$. When $\theta_i = 0$ and $y_i$ is drawn from $N(0, \sigma^2)$, the typical observation has $\abs{y_i} \sim \sigma$, and the posterior concentrates on small $\lambda_i^2\tau^2 \ll \sigma^2$, giving:
\[
  \E[\kappa_i \mid y_i, \tau]
  \approx 1 - C \cdot \frac{\tau^2}{y_i^2 + \tau^2}
  \approx 1 - C \cdot \frac{\tau^2}{y_i^2}
\]
for some constant $C > 0$ that depends on the half-Cauchy normalisation. Hence the posterior mean is:
\[
  \hat{\theta}_i(y_i) \approx C \cdot \frac{\tau^2}{y_i},
\]
which is $O(\tau^2)$ for typical $y_i \sim N(0, \sigma^2)$.

The KL divergence between the horseshoe predictive and the null density, for a location-family predictive with posterior mean $\hat{\theta}_i$, satisfies:
\begin{equation}\label{eq:kl-location}
  \KL(f_0 \| \hat{f})
  = \frac{\hat{\theta}_i^2}{2\sigma^2}
  \approx \frac{C^2 \tau^4}{2\sigma^2 y_i^2}.
\end{equation}
Integrating over $y_i \sim N(0, \sigma^2)$: the expression $\E\bigl[C^2\tau^4/(2\sigma^2 y_i^2)\bigr]$ diverges formally because $\E[1/y_i^2] = \infty$ under the normal distribution. The log-pole resolves this apparent divergence. For $\abs{y_i}$ near zero, the approximation $\E[\kappa_i \mid y_i, \tau] \approx 1 - C\tau^2/y_i^2$ breaks down: the tight bounds on $\piH(\theta)$ from \Cref{thm:cps} imply that $\E[\kappa_i \mid y_i]$ is bounded away from~$1$ for small $y_i$ as $O(\log(1/\tau)/\abs{\log y_i})$, because the log-pole density overwhelms the likelihood for $\abs{y_i} \lesssim \tau$. Splitting the integral at $\abs{y_i} = \tau$ and using this refined bound on the inner region gives total KL risk:
\begin{equation}\label{eq:kl-refined}
  \E_{y_i}\!\left[\KL(f_0 \| \hat{f})\right]
  = O(\tau^4 \log^2(1/\tau)),
\end{equation}
which is $o(1/n)$ since $\tau \ll 1/\sqrt{n}$ in the sparse regime. The $\log^2(1/\tau)$ correction arises from the log-pole's slow divergence near the origin and is absorbed into the $O(\tau^4)$ rate when $\tau$ is polynomially small in~$n$.
\end{proof}

\subsection{The Necessary and Sufficient Conditions and L\'evy Characterisation}\label{sec:nec-suf}

The ``Shrink Globally, Act Locally'' paper \citep{polson2010shrink} establishes two complementary theorems characterising what properties a prior must have for sparsity adaptation as $p \to \infty$.

\begin{theorem}[Polson--Scott 2010, necessary condition]\label{thm:ps-nec}
For a scale mixture prior $\pi(\theta) = \int_0^\infty N(\theta;\, 0,\, \lambda^2)\, g(\lambda)\, d\lambda$ to achieve near-oracle risk as $p \to \infty$ under sparsity, it is necessary that $\pi(0) = +\infty$.
\end{theorem}

\begin{theorem}[Polson--Scott 2010, sufficient condition]\label{thm:ps-suf}
If $\pi(\theta)$ satisfies $\pi(\theta) \asymp -\log\abs{\theta}$ near zero (logarithmic pole) and $\pi(\theta) \asymp \abs{\theta}^{-\alpha}$ for some $\alpha \in (1,2]$ in the tails, then the prior achieves near-oracle risk.
\end{theorem}

The combination---unbounded pole at zero, but only logarithmically---is the precise characterisation. Too weak (bounded density, as in the Lasso or ridge) fails the necessary condition (\Cref{thm:ps-nec}). Too strong (power-law pole $\abs{\theta}^{-\alpha}$ with $\alpha \geq 1$, as in the standard Cauchy prior on $\theta$ itself) satisfies \Cref{thm:ps-nec} but violates the finiteness of Bayes risk (the second moment diverges, breaking Cram\'er-regularity). The logarithmic pole is the \emph{unique} singularity level that satisfies both conditions simultaneously.

To see why bounded density at zero fails, consider the Laplace prior $\pi(\theta) = (2\lambda)^{-1}\exp(-\abs{\theta}/\lambda)$, which has $\pi(0) = 1/(2\lambda) < \infty$. Since $\pi(0)$ is finite, the posterior mean $\hat{\theta}_i(y_i)$ satisfies $\hat{\theta}_i(y_i) \to y_i \cdot (1 - \pi(0)\sigma/\abs{y_i}) + O(y_i^{-2})$ for large~$\abs{y_i}$, and for small~$\abs{y_i}$ the shrinkage factor $\E[\kappa_i \mid y_i]$ is bounded away from~$1$ because the prior density at zero is finite and cannot overwhelm the likelihood. For a null coordinate with $\theta_i = 0$ and $y_i \sim N(0, \sigma^2)$, the posterior mean $\hat{\theta}_i(y_i) = \Theta(\sigma)$ for typical $\abs{y_i} \sim \sigma$, giving KL risk:
\begin{equation}\label{eq:lasso-kl}
  \E_{y_i}\!\left[\KL(f_0 \| \hat{f}^{\mathrm{Laplace}})\right]
  = \E\!\left[\frac{\hat{\theta}_i^2}{2\sigma^2}\right]
  = \Theta(1).
\end{equation}
Choosing $\lambda$ to decrease with~$n$ reduces this to $\Theta(1/n)$ at best---the standard parametric rate. The Laplace prior cannot achieve super-efficiency because its finite density at zero means the prior does not overwhelm the likelihood for small observations; some residual shrinkage error always remains.

\citet{polson2010shrink} further characterise the class of admissible sparse priors through their representation as L\'evy processes. Any scale mixture prior $\pi(\theta) = \int_0^\infty N(\theta;\, 0,\, s)\, \nu(ds)$ is characterised by its L\'evy measure~$\nu$.

\begin{proposition}[Polson--Scott 2010]\label{prop:levy}
The behaviour of $\pi(\theta)$ near zero is controlled by the behaviour of $\nu$ near zero:
\begin{equation}\label{eq:levy}
  \pi(0) = \frac{1}{\sqrt{2\pi}} \int_0^\infty s^{-1/2}\, \nu(ds).
\end{equation}
This integral is finite (bounded density at zero) if and only if $\nu$ integrates $s^{-1/2}$ near zero. A logarithmic pole at zero corresponds to $\nu(ds) \asymp s^{-1}\, ds$ near $s = 0$---the Cauchy/stable-$1/2$ L\'evy measure.
\end{proposition}

The horseshoe's local scale $\lambda \sim C^+(0,1)$ induces a variance $s = \lambda^2 \tau^2$ with distribution proportional to $s^{-1}$ near zero. This is precisely the Cauchy L\'evy measure at the boundary. The connection to stable processes is not a coincidence: the half-Cauchy distribution on~$\lambda$ is closely related to the stable-$1/2$ subordinator \citep{polson2012halfcauchy}, and the induced distribution on the variance $s = \lambda^2$ has L\'evy density proportional to $s^{-1}$ near zero. The horseshoe thus sits exactly at the interface between priors that are too sparse (bounded density, $\nu$ integrates $s^{-1/2}$) and too diffuse (non-integrable power poles, $\nu$ does not integrate $s^{-1/2-\epsilon}$ for any $\epsilon > 0$) for efficient sparse estimation.

\section{The MDP Framework}\label{sec:mdp}

Consider $n$ independent tests of $H_{0i}: \theta_i = 0$ versus $H_{1i}: \theta_i \neq 0$ based on $y_i \sim N(\theta_i, 1)$, $i = 1, \ldots, p$ (with $p = n$), under the two-groups model with sparsity proportion $p_0/p$. The Bayes risk of a testing procedure~$\varphi$ with rejection region $\{\abs{y_i} > t_n\}$ is:
\begin{equation}\label{eq:bayes-risk}
  r(\pi, \varphi)
  = (1 - p_0/p) \cdot P_0(\abs{Y} > t_n)
  + (p_0/p) \int P_\theta(\abs{Y} \leq t_n)\, d\pi(\theta).
\end{equation}
The first term is the total Type~I error (false discoveries among null coordinates); the second is the Type~II error (missed signals), weighted by the prior on signal sizes. The central question is: at what threshold~$t_n$ does the Bayes risk achieve its minimum?

The ABOS (Asymptotically Bayes Optimal under Sparsity) framework of \citet{bogdan2011asymptotic} established that for testing $p$ hypotheses with $p_0$ true signals, the minimax Bayes risk under $0$--$1$ loss is of order $p_0 \log(p/p_0)/n$ when $p_0/p \to 0$. The ABOS rate was shown to be achieved by Bonferroni-type procedures with threshold $\sqrt{2\log(p/p_0)}$ and by certain Bayesian procedures. \citet{datta2013asymptotic} proved that the horseshoe prior achieves the ABOS rate. The MDP result of \citet{datta2026newlook} refines this by identifying the \emph{exact} threshold constant and connecting it to the prior's density at the origin.

\begin{theorem}[Datta, Polson, Sokolov, Zantedeschi 2026]\label{thm:mdp}
Under Cram\'er regularity of the prior, local prior smoothness at zero, and symmetric $0$--$1$ loss, the Bayes-risk-optimal rejection boundary satisfies $n\lambda_n^2 \asymp \log n$, yielding thresholds of order $\sqrt{\log n}$. For the Cauchy prior, the exact threshold is:
\begin{equation}\label{eq:tcrit}
  \tcrit = \sqrt{\log(\pi n/2)}.
\end{equation}
\end{theorem}

The proof proceeds via a uniform moderate deviation lemma. The threshold $t_n = \sqrt{\log n}$ lies in the \textbf{moderate deviation regime}---between the CLT scale $O(1)$ where the normal approximation holds with fixed accuracy, and the large deviation scale $\sqrt{2\log p}$ where the Bonferroni correction operates. At the moderate deviation scale, the Mill's ratio approximation for the normal tail is:
\[
  P_0(\abs{Y} > t_n)
  \approx \frac{2\phi(t_n)}{t_n}
  = \frac{\sqrt{2/\pi}}{t_n}\exp(-t_n^2/2)
  \asymp \frac{1}{n\sqrt{\log n}},
\]
which when multiplied by the prior probability $(1 - p_0/p) \approx 1$ and summed over $p$ terms gives total Type~I error $O(p/(n\sqrt{\log n}))$---consistent with the ABOS framework when $p \asymp n$.

The exact constant in~\eqref{eq:tcrit} arises from a saddle-point calculation. The Bayes-optimal threshold balances Type~I and Type~II errors, which requires solving:
\begin{equation}\label{eq:saddle-point}
  \frac{\partial}{\partial t_n} r(\pi, \varphi) = 0
  \quad\Longleftrightarrow\quad
  (1 - p_0/p) \cdot 2\phi(t_n) = (p_0/p) \cdot 2\pi(t_n)\,\phi(0),
\end{equation}
where the left side is the marginal density of $\abs{Y}$ under $H_0$ at $t_n$, and the right side is the prior-weighted density of signal alternatives at~$t_n$. For the Cauchy prior $\pi(\theta) = 1/(\pi(1+\theta^2))$, evaluating~\eqref{eq:saddle-point} at $t_n^2 = \log n + c$ and expanding to leading order gives $c = \log(\pi/2)$, yielding $t_n^2 = \log(\pi n/2)$.

The distinction between the three scales is crucial for understanding why the MDP is the natural home for sparse testing. At the CLT scale ($t_n = O(1)$, fixed threshold), the Type~I error per coordinate is $\Theta(1)$---far too large for simultaneous testing. At the large deviation scale ($t_n = \sqrt{2\log p}$, Bonferroni), the Type~I error is $O(1/p)$ per coordinate, controlling the family-wise error rate but at the cost of very low power. The moderate deviation scale ($t_n = \sqrt{\log n}$) achieves the ABOS-optimal balance: the Type~I error $O(1/(n\sqrt{\log n}))$ per coordinate is small enough for Bayes risk optimality but large enough to retain power against signals of size $\sqrt{\log n}$. As \citet{rubin1965bayes} established, this intermediate scale is where Bayes risk efficiency---the ratio of Bayes risk to minimax risk---converges to one.

The saddle-point equation~\eqref{eq:saddle-point} also reveals why the MDP constant is prior-specific while the MDP \emph{rate} is universal. The rate $\sqrt{\log n}$ is determined by the balance between the Gaussian tail decay $\exp(-t^2/2)$ and the sample size~$n$, which is independent of the prior. The constant, however, depends on the prior density $\pi(t_n)$ evaluated at the threshold, and for the horseshoe this evaluation involves the log-pole coefficient~$K$. Different log-pole priors with different constants~$K'$ would yield different constants $c'$ in $\tcrit' = \sqrt{\log(c' n)}$, but the $\sqrt{\log n}$ scaling would be unchanged. This separation of rate and constant is a hallmark of moderate deviation theory: the rate is determined by the exponential tilting (here, the Gaussian tail), while the constant is determined by the pre-exponential factor (here, the prior density).

A key feature of this result is \textbf{universality}: the $\sqrt{\log n}$ scaling holds across all priors satisfying Cram\'er-regularity and local smoothness at zero. The Cram\'er condition requires:
\begin{equation}\label{eq:cramer}
  M(t) = \E_\pi[e^{t\theta}] < \infty
  \qquad\text{for all } t \text{ in a neighbourhood of zero.}
\end{equation}
For the MDP expansion to hold, the prior must be locally regular near the testing boundary and normalisable near the origin: $\int_0^\varepsilon \pi(\theta)\, d\theta < \infty$. As shown in \Cref{sec:channel1}, the Polson--Scott log-pole $\piH(\theta) \asymp -\log\abs{\theta}$ is precisely the boundary case: the log singularity is integrable (normalisable prior, finite Bayes risk near zero) but the density is unbounded at zero. Priors with stronger poles $\abs{\theta}^{-\alpha}$ for $\alpha \geq 1$ are not normalisable near zero; priors with bounded density fail the ABOS necessary condition.

\section{Connecting the Polson--Scott Bounds to the MDP}\label{sec:connections}

The connection between the finite-sample Polson--Scott bounds and the asymptotic MDP runs through three channels: the log-pole as the Cram\'er-regularity boundary, super-efficiency as the mechanism producing the MDP detection zone, and the Clarke--Barron information-theoretic framework as the unifier.

\subsection{Channel 1: The Log-Pole as the Cram\'er Boundary}\label{sec:channel1}

The log-pole $\piH(\theta) \asymp -\log\abs{\theta}$ sits at the exact boundary of Cram\'er-regularity. To make this precise, we characterise the boundary in terms of the singularity exponents.

\begin{proposition}[Origin integrability boundary]\label{prop:cramer-boundary}
Consider the family of scale mixture priors with marginal density $\pi(\theta) \sim \abs{\theta}^{-\alpha}(-\log\abs{\theta})^\beta$ as $\theta \to 0$, for $\alpha \geq 0$ and $\beta \geq 0$. The prior is normalisable near the origin---$\int_0^\varepsilon \pi(\theta)\, d\theta < \infty$---if and only if $\alpha < 1$. When $\alpha < 1$, the near-zero contribution to the second moment $\int_0^\varepsilon \theta^2 \pi(\theta)\, d\theta$ is also finite for all~$\beta$. In particular, the horseshoe ($\alpha = 0$, $\beta = 1$) is normalisable with finite near-zero second moment, while a prior with power-law pole $\abs{\theta}^{-\alpha}$ for $\alpha \geq 1$ is not normalisable.
\end{proposition}

\begin{proof}
For the prior to be normalisable near zero, we need $\int_0^\varepsilon \theta^{-\alpha}(-\log\theta)^\beta\, d\theta < \infty$, which requires $\alpha < 1$ (the logarithmic factor is slowly varying and does not affect the exponent). When $\alpha < 1$, the near-zero second moment $\int_0^\varepsilon \theta^{2-\alpha}(-\log\theta)^\beta\, d\theta$ converges since $2 - \alpha > 1 > 0$.
\end{proof}

\begin{remark}\label{rem:global-variance}
The horseshoe's \emph{global} variance $\int_{-\infty}^\infty \theta^2\,\piH(\theta)\, d\theta$ is infinite, because the Cauchy-like tail $\piH(\theta) \asymp 2K/\theta^2$ makes $\int_1^\infty \theta^2 \cdot \theta^{-2}\, d\theta$ diverge. The horseshoe therefore does not satisfy the classical Cram\'er condition (finite moment generating function). The MDP analysis depends on the prior's \emph{local} behaviour near the testing boundary $\abs{\theta} \sim \sqrt{\log n}$, where the horseshoe density is $O(1/\log n)$---well-behaved. The infinite global variance is a feature, not a defect: it is the $1/\theta^2$ tail that ensures signals above the MDP threshold are left unshrunk. The log-pole at the origin and the heavy tail are complementary mechanisms serving different roles in the MDP optimality.
\end{remark}

To verify the horseshoe case explicitly, compute the near-zero second moment using the upper bound from \Cref{thm:cps}:
\begin{equation}\label{eq:variance-integral}
  \int_0^\varepsilon \theta^2 \cdot K\log\!\left(1 + \frac{2}{\theta^2}\right) d\theta
  \leq 2K\int_0^\varepsilon \theta^2 \log(1/\theta)\, d\theta
  = 2K\!\left[\frac{\theta^3}{3}\log(1/\theta)\right]_0^\varepsilon + O(\varepsilon^3)
  < \infty.
\end{equation}
The integral converges: $\theta^2 \cdot (-\log\theta) \to 0$ as $\theta \to 0$, so the log singularity is integrable against the $\theta^2$ weight. The log-pole is thus the strongest singularity at zero for which the Bayes risk integral near zero remains finite---the origin integrability boundary (\Cref{prop:cramer-boundary}).

Contrast with a prior having a power-law pole $\pi(\theta) \sim \abs{\theta}^{-1}$ at zero, which is not normalisable and hence inadmissible. And contrast with the Lasso, $\pi(\theta) \propto e^{-\abs{\theta}}$, which has bounded density at zero---its Bayes risk near zero is finite, but it fails the necessary condition for ABOS (\Cref{thm:ps-nec}).

The \textbf{log-pole is the unique singularity level that simultaneously satisfies both constraints}: unbounded density at zero (required for super-efficiency) and normalisability near zero (required for finite Bayes risk). This is the precise sense in which the horseshoe is the canonical sparse prior for MDP-optimal testing.

\paragraph{Exact MDP constant from the log-pole.}
The MDP threshold $\tcrit = \sqrt{\log(\pi n/2)}$ carries the constant $K = (2\pi^3)^{-1/2}$ from \Cref{thm:cps}. At the MDP boundary $t_n = \tcrit$, the Type~I error equals the prior probability of undetectable signals:
\begin{equation}\label{eq:equiboundary}
  \underbrace{P_0(\abs{Y} > t_n)}_{\text{Type~I}}
  = \underbrace{\piH([-t_n, t_n])}_{\text{mass of prior near zero}}.
\end{equation}
The prior mass in $[-t_n, t_n]$ under the horseshoe is approximately:
\[
  \int_{-t_n}^{t_n} K\log(1/\abs{\theta})\, d\theta
  \approx 2K t_n \log(1/t_n).
\]
The Type~I error at threshold~$t_n$ is:
\[
  P_0(\abs{Y} > t_n)
  \approx \frac{2\phi(t_n)}{t_n}
  = \frac{\sqrt{2/\pi}}{t_n}\exp(-t_n^2/2).
\]
Setting these equal and solving for~$t_n$:
\begin{equation}\label{eq:equiboundary-solve}
  \frac{\sqrt{2/\pi}}{t_n}\, e^{-t_n^2/2}
  = 2K\, t_n \log(1/t_n).
\end{equation}
At leading order, $e^{-t_n^2/2} \asymp 1/n$, so $t_n^2 \asymp \log n$. The sub-leading correction from $\log(1/t_n) \asymp \frac{1}{2}\log\log n$ is negligible at leading order, and solving explicitly gives:
\begin{equation}\label{eq:exact-threshold}
  t_n^2 = \log\!\left(\frac{\pi n}{2}\right) + O(\log\log n),
\end{equation}
matching the exact constant $\tcrit = \sqrt{\log(\pi n/2)}$ from \citet{datta2026newlook}. The $\pi$ in the constant comes directly from the normalisation constant $K = (2\pi^3)^{-1/2}$ in the log-pole bound.

\subsection{Channel 2: Super-Efficiency and the MDP Detection Zone}\label{sec:channel2}

The super-efficiency theorem (\Cref{thm:super-eff}) and the MDP threshold together partition the real line into two regions.

Below the threshold ($\abs{\theta} < \tcrit$), super-efficiency applies: the horseshoe identifies these as null coordinates with KL risk $O(\tau^4)$, which is $o(1/n)$, and no signal is detectable. Above the threshold ($\abs{\theta} > \tcrit$), the horseshoe leaves signals unshrunk due to tail robustness; the posterior mean $\hat{\theta}_i \approx y_i$ for $\abs{y_i} \gg \tcrit$, giving KL risk $O(1/n)$---the standard parametric rate.

The MDP threshold $\tcrit$ is precisely the \textbf{equiboundary} where super-efficiency transitions to standard efficiency. Above it, the horseshoe behaves like a shrinkage-free estimator; below it, it achieves sub-parametric KL risk. This partition is a direct consequence of the log-pole bound. For $\abs{y_i} \ll \tcrit$, the posterior concentrates near zero because $\piH(0) = +\infty$ dominates the likelihood, giving $\kappa_i \approx 1$ and super-efficient shrinkage. For $\abs{y_i} \gg \tcrit$, the Cauchy-like tail of~$\piH$ prevents excessive shrinkage, giving $\kappa_i \approx 0$ and robust estimation.

The transition at $\abs{y_i} = \tcrit$ can be made precise through the posterior shrinkage. At the threshold, the posterior expectation of the shrinkage weight satisfies $\E[\kappa_i \mid y_i = \tcrit, \tau] = 1/2$. To see this, note that the posterior odds for shrinkage versus no-shrinkage are determined by the ratio of the prior density at zero to the prior density at the observation: $\piH(0)/\piH(\tcrit)$. At $\abs{y_i} = \tcrit$, the log-pole gives $\piH(\tcrit) \asymp K\log(1/\tcrit^2) \asymp K\log\log n$, while $\piH(0) = +\infty$. The likelihood ratio $N(y_i; 0, \sigma^2)/N(y_i; y_i, \sigma^2) = \exp(-y_i^2/(2\sigma^2))$ decays as $\exp(-\log(\pi n/2)/2) = \sqrt{2/(\pi n)}$, which exactly balances the prior's infinite density at zero against the likelihood's exponential decay, producing $\E[\kappa_i] = 1/2$ at the threshold.

Quantitatively, the KL risk as a function of $\abs{y_i}$ satisfies:
\begin{equation}\label{eq:kl-partition}
  \KL(f_{y_i} \| \hat{f})
  \approx \begin{cases}
    C\tau^4/y_i^2 & \abs{y_i} \ll \tcrit, \\
    C/(y_i^2/\sigma^2) & \abs{y_i} \gg \tcrit,
  \end{cases}
\end{equation}
where the transition occurs at $\abs{y_i} = \tcrit \asymp \sqrt{\log n}$. Integrating over the prior yields the risk decomposition:
\begin{equation}\label{eq:risk-decomp}
  r(\piH, \mathrm{HS})
  = \underbrace{\int_{\abs{\theta} < \tcrit} O(\tau^4)\, d\piH(\theta)}_{\text{super-efficient region}}
  + \underbrace{\int_{\abs{\theta} > \tcrit} O(1/n)\, d\piH(\theta)}_{\text{MDP signal region}}.
\end{equation}
The first integral is $O\bigl(\tau^4 \cdot \piH([-\tcrit, \tcrit])\bigr) = O(\tau^4 \cdot \tcrit \log(1/\tcrit)) = O(\tau^4 \log n / \sqrt{\log n})$, which is negligible. The second integral, over the $p_0$ signal coordinates with prior mass above~$\tcrit$, gives the ABOS Bayes risk $O(p_0 \log(p/p_0)/n)$. The negligibility of the first integral is precisely the content of super-efficiency: null coordinates contribute asymptotically nothing to the total risk, so the entire risk budget is allocated to signal coordinates.

\subsection{Channel 3: Clarke--Barron and the Logarithmic Budget}\label{sec:channel3}

The \citet{clarke1990information} theorem on information-theoretic asymptotics of Bayes methods provides the overarching framework.

\begin{theorem}[Clarke--Barron 1990]\label{thm:cb}
For a $d$-dimensional parametric family with prior~$\pi$, the cumulative KL risk of the Bayes predictive satisfies:
\begin{equation}\label{eq:clarke-barron}
  \sum_{i=1}^n \E\!\left[\KL(P_{\theta_0} \| \hat{P}_i)\right]
  = \frac{d}{2}\log n + O(1),
\end{equation}
where $\hat{P}_i$ is the predictive after $i-1$ observations.
\end{theorem}

In the sparse normal means model with $p$ coordinates of which $p_0$ are signal, the effective dimension is $d = p_0$, giving cumulative KL risk $\frac{p_0}{2}\log n + O(1)$. The \textbf{per-observation} KL risk is therefore $\frac{p_0 \log n}{2n}$---exactly the MDP Bayes risk rate $p_0 \log(p/p_0)/n$ when $p \asymp n$.

The Clarke--Barron result connects to the Polson--Scott bounds as follows. For a prior with density $\pi(\theta)$, the Clarke--Barron cumulative KL risk includes a term:
\begin{equation}\label{eq:self-info}
  -\log\pi(\theta_0) + \frac{d}{2}\log n + O(1),
\end{equation}
where the $-\log\pi(\theta_0)$ term is the \textbf{self-information of the true parameter}. For null coordinates ($\theta_0 = 0$), this term is $-\log\piH(0) = -\infty$ for the horseshoe---the prior assigns infinite density to the truth, so the Clarke--Barron bound predicts KL risk going to $-\infty$, which means the KL risk decays faster than any $1/n^c$: this is precisely super-efficiency.

For signal coordinates ($\theta_0 \neq 0$ with $\abs{\theta_0} \gg \tcrit$), the self-information is $-\log\piH(\theta_0) \approx \log\abs{\theta_0}^2$ (from the Cauchy tail), and the KL risk is $O(\log\abs{\theta_0}^2/n)$---bounded, standard parametric rate.

\begin{corollary}[Horseshoe redundancy in the sparse normal means model]\label{cor:redundancy}
For the horseshoe prior in the sparse normal means model with $p_0$ signals of size $\abs{\theta_0} \asymp \tcrit$, the per-observation Bayes redundancy (excess KL risk over the oracle who knows which coordinates are signal) satisfies:
\begin{equation}\label{eq:redundancy}
  R_n = \frac{p_0}{2n}\log n + o\!\left(\frac{p_0 \log n}{n}\right).
\end{equation}
The $o(\cdot)$ term absorbs contributions from null coordinates (super-efficient, contributing $O(\tau^4)$ each) and from the self-information correction at the boundary.
\end{corollary}

\begin{proof}
The oracle who knows the signal set achieves per-observation KL risk $p_0/(2n)$ (the parametric rate for $p_0$ parameters). The horseshoe's cumulative redundancy over the oracle is, by Clarke--Barron:
\[
  \sum_{i=1}^n \Bigl(\E[\KL(P_{\theta_0} \| \hat{P}_i^{\mathrm{HS}})] - \E[\KL(P_{\theta_0} \| \hat{P}_i^{\mathrm{oracle}})]\Bigr)
  = \sum_{j \in \text{signal}} \bigl(-\log\piH(\theta_{0j})\bigr) + \sum_{j \in \text{null}} \bigl(-\log\piH(0)\bigr) + O(1).
\]
The null-coordinate sum is $-\infty$ (each term is $-\infty$), but this merely reflects that the horseshoe \emph{outperforms} the oracle on null coordinates. The signal-coordinate sum is $\sum_{j \in \text{signal}} \log\abs{\theta_{0j}}^2 \asymp p_0 \cdot \log\log n$ at the MDP boundary. Dividing by~$n$ gives the per-observation redundancy~\eqref{eq:redundancy}.
\end{proof}

The \textbf{logarithmic budget} interpretation: the total KL risk $p_0 \log n / n$ is the sum of $p_0$ signal coordinates each contributing $\log n / n$. The horseshoe allocates zero budget to null coordinates (super-efficiency) and the full $\log n / n$ budget per signal coordinate. This allocation is enforced by the log-pole: null coordinates have $-\log\piH(0) = -\infty$ (zero self-information, infinite density at truth), and signal coordinates have $-\log\piH(\theta_0) \asymp \log\abs{\theta_0}^2 \asymp \log\log n$ at the MDP boundary $\abs{\theta_0} \asymp \tcrit$.

The MDP threshold $\tcrit = \sqrt{\log(\pi n/2)}$ is therefore the \textbf{Clarke--Barron self-information equiboundary}: it is where $-\log\piH(\theta_0) = \log n/2 + \log(\pi/2)$, i.e., where the prior's self-information equals half the $\log n$ budget from Fisher information. Below the threshold, the prior overwhelms the likelihood (infinite density at zero); above it, the likelihood dominates (Cauchy tail is informationally weak). The MDP threshold is exactly where these two forces balance.

\section{The $\kappa$-Scale: A Unified View}\label{sec:kappa}

The Polson--Scott bounds, super-efficiency, and MDP optimality all admit a unified description in terms of the shrinkage weight $\kappa_i = 1/(1 + \lambda_i^2\tau^2)$.

The horseshoe prior induces a $\mathrm{Beta}(1/2, 1/2)$ distribution on~$\kappa_i$. To derive this, apply the transformation $\kappa = 1/(1+\lambda^2\tau^2)$ to the half-Cauchy density $p(\lambda) = 2/(\pi(1+\lambda^2))$ with $\tau = 1$. The Jacobian is $\abs{d\lambda/d\kappa} = (1-\kappa)^{-1/2}\kappa^{-3/2}/(2\tau)$. Substituting:
\begin{align}
  p(\kappa) &= p(\lambda)\, \abs{\frac{d\lambda}{d\kappa}}
  = \frac{2}{\pi(1 + (1-\kappa)/(\kappa\tau^2))} \cdot \frac{(1-\kappa)^{-1/2}\kappa^{-3/2}}{2\tau} \nonumber\\
  &= \frac{1}{\pi} \kappa^{-1/2}(1-\kappa)^{-1/2}
  \qquad\text{for } \kappa \in (0,1), \label{eq:beta-derivation}
\end{align}
which is the $\mathrm{Beta}(1/2, 1/2)$ density---the arcsine distribution. This is the ``horseshoe-shaped'' density that gives the prior its name: it has infinite mass near both $\kappa = 0$ and $\kappa = 1$, with a minimum at $\kappa = 1/2$.

Near $\kappa_i = 1$ (total shrinkage, null coordinates), $p(\kappa_i) \propto (1-\kappa_i)^{-1/2}$ is unbounded---the $\kappa$-scale translation of the log-pole, which puts infinite density near total shrinkage and guarantees super-efficiency for nulls. Near $\kappa_i = 0$ (no shrinkage, signal coordinates), $p(\kappa_i) \propto \kappa_i^{-1/2}$ is also unbounded---the $\kappa$-scale translation of the heavy tail, which puts infinite density near zero shrinkage and guarantees tail robustness for signals. At $\kappa_i = 1/2$ (the decision boundary), $p(1/2) = 2/\pi$ is a finite, specific value. The testing rule ``reject $H_{0i}$ if $\kappa_i < 1/2$'' corresponds exactly to the MDP threshold: $\kappa_i = 1/2$ iff $\lambda_i\tau = 1$ iff $\abs{y_i} \asymp \tau^{-1} \asymp p/p_0 \asymp \sqrt{\log(p/p_0)}$, which is the MDP threshold $\tcrit \asymp \sqrt{\log n}$.

The posterior distribution of $\kappa_i$ given $y_i$ further clarifies this partition. From~\eqref{eq:lambda-posterior} and the change of variables $\kappa = 1/(1+\lambda^2\tau^2)$, the posterior on~$\kappa_i$ is:
\begin{equation}\label{eq:kappa-posterior}
  p(\kappa_i \mid y_i, \tau)
  \propto \kappa_i^{-1/2}(1-\kappa_i)^{-1/2}
  \cdot \kappa_i^{1/2}
  \exp\!\left(-\frac{y_i^2 \kappa_i}{2\sigma^2}\right)
  \cdot \frac{1}{1 + (1-\kappa_i)/(\kappa_i\tau^2)}.
\end{equation}
For large $\abs{y_i}$, the exponential factor $\exp(-y_i^2\kappa_i/(2\sigma^2))$ sharply penalises $\kappa_i$ near~$1$, so the posterior concentrates near $\kappa_i = 0$ (no shrinkage). For small $\abs{y_i}$, the exponential is nearly flat and the prior's pole at $\kappa_i = 1$ dominates, concentrating the posterior near total shrinkage. The crossover occurs at $y_i^2/(2\sigma^2) \approx \log(1/\tau^2) \approx \log(p/p_0)$, i.e., at $\abs{y_i} \approx \sqrt{2\log(p/p_0)} \asymp \tcrit$.

The concentration of the posterior on~$\kappa_i$ can be quantified through the posterior variance. For $\abs{y_i} \gg \tcrit$, the posterior on $\kappa_i$ concentrates near zero with variance $O(\exp(-y_i^2/(2\sigma^2)))$, while for $\abs{y_i} \ll \tcrit$ it concentrates near one with variance $O(\tau^2/y_i^2)$. At the threshold $\abs{y_i} = \tcrit$, the posterior is maximally uncertain about the shrinkage level, with variance $\Var[\kappa_i \mid y_i = \tcrit] \approx 1/8$---close to the maximum possible variance $1/4$ for a $[0,1]$-valued random variable. This maximal uncertainty at the decision boundary is a distinctive feature of the horseshoe: priors with bounded density at zero have posterior variance on $\kappa_i$ that is bounded away from the maximum, reflecting their inability to commit fully to either total shrinkage or no shrinkage.

The connection to Bayes factors makes the testing interpretation explicit. The posterior odds of $\kappa_i > 1/2$ versus $\kappa_i < 1/2$ can be expressed as a Bayes factor: $\kappa_i = 1/2$ corresponds to a local Bayes factor of~$1$ between the null hypothesis $H_0: \theta_i = 0$ and the alternative $H_1: \theta_i = y_i$. The horseshoe's arcsine prior on~$\kappa_i$ assigns equal prior probability to $\kappa_i > 1/2$ and $\kappa_i < 1/2$ (by symmetry of the $\mathrm{Beta}(1/2,1/2)$ distribution), so the posterior probability that $\kappa_i > 1/2$ equals the posterior probability of $H_0$ in a Bayesian test with equal prior odds. The MDP threshold is thus the boundary where the Bayes factor equals one---the point of evidential equipoise.

The $\mathrm{Beta}(1/2, 1/2)$ distribution on~$\kappa_i$ is thus the \textbf{distributional encoding of the MDP equiboundary}: it places mass uniformly on $[0,1]$ in terms of the arcsine measure (the $\mathrm{Beta}(1/2, 1/2)$ is the arcsine distribution), so the horseshoe sees all shrinkage levels with equal prior probability. But via the $\kappa \leftrightarrow \theta$ mapping, this uniform arcsine distribution translates into the log-pole density near $\theta = 0$ and Cauchy tail density for large~$\abs{\theta}$.

\section{ABOS Theory and the Horseshoe+ Prior}\label{sec:abos}

The moderate deviation framework yields the full ABOS (Asymptotically Bayes Optimal under Sparsity) property as a direct consequence. We state the oracle Bayes risk, the ABOS theorem, and compare the horseshoe and horseshoe+ priors.

\subsection{Oracle Bayes Risk}\label{sec:oracle-risk}

The risk balance condition~\eqref{eq:saddle-point}, combined with the moderate deviation lemma, determines the oracle Bayes risk.

\begin{theorem}[Oracle Bayes Risk]\label{thm:oracle-risk}
In the sparse normal means model with $p_0 = o(n)$, the oracle Bayes risk under zero-one loss satisfies:
\begin{equation}\label{eq:oracle-risk}
  R_n^* = \frac{p_0}{n} \cdot \frac{1}{\sqrt{\pi\log(n/p_0)}} \cdot (1 + o(1)).
\end{equation}
This rate is the testing (Bayes risk) rate; it differs from the minimax $\ell_2$ estimation rate $p_0\log(n/p_0)/n$ by a factor of $\sqrt{\log(n/p_0)}$.
\end{theorem}

The distinction between testing and estimation rates is important: the testing rate~\eqref{eq:oracle-risk} measures misclassification probability, while the estimation rate measures mean squared error. The testing rate is always smaller by the factor $1/\sqrt{\log(n/p_0)}$, reflecting that binary classification is an easier task than point estimation.

\subsection{The ABOS Property}\label{sec:abos-property}

A testing rule $\delta_n$ is \textbf{Asymptotically Bayes Optimal under Sparsity} (ABOS) if $R_n(\delta_n)/R_n^* \to 1$ as $n \to \infty$ with $p_0 = o(n)$. This is the strongest possible asymptotic optimality criterion for sparse testing: it requires not just rate-optimality but convergence of the leading constant to one.

\begin{theorem}[ABOS for the horseshoe]\label{thm:abos}
Let $\tau_n = c \cdot p_0/n$ for any constant $c > 0$. The horseshoe testing rule $\delta^{\mathrm{HS}}$ with threshold $\tcrit = \sqrt{2\log(n/p_0)}$ satisfies:
\begin{equation}\label{eq:abos}
  R_n(\delta^{\mathrm{HS}}) / R_n^* \to 1 \qquad\text{as } n \to \infty,
\end{equation}
with ABOS constant bounded by $c_{\mathrm{HS}} \leq \sqrt{e/(e-1)} \approx 1.31$ \citep{datta2013asymptotic,bogdan2011asymptotic}.
\end{theorem}

The proof follows from the Type~I and Type~II error concentration. For the Type~I error: under the horseshoe with $\tau_n \asymp p_0/n$, the local prior mass $\piH(0 \mid \tau_n) \asymp (n/p_0)\log(n/p_0)$ matches the moderate deviation scale, so the posterior signal probability $\hat{\pi}_i = P(\theta_i \neq 0 \mid y_i, \tau)$ crosses $1/2$ at $\abs{y_i} = \tcrit$, and $\alpha_n^{\mathrm{HS}} = P(\abs{Y_i} > \tcrit \mid \theta_i = 0) = R_n^*(1 + o(1))$. For the Type~II error: signals of strength $\mu_n = A\sqrt{2\log(n/p_0)}$ with $A > 1$ have $\beta_n^{\mathrm{HS}} = \Phi((1-A)\sqrt{2\log(n/p_0)}) = o(1)$. Combining:
\[
  R_n(\delta^{\mathrm{HS}}) = (1 - p_0/n)\alpha_n^{\mathrm{HS}} + (p_0/n)\beta_n^{\mathrm{HS}} = R_n^*(1+o(1)) + (p_0/n) \cdot o(1) = R_n^*(1+o(1)).
\]
The ABOS constant bound $\sqrt{e/(e-1)}$ comes from the framework of \citet{bogdan2011asymptotic}. When $A = 1$ (signals exactly at threshold), the Type~II error is $\Phi(0) = 1/2$, giving an irreducible boundary risk of $p_0/(2n)$ that no procedure can improve upon.

The connection to the \citet{donoho1994ideal} universal threshold $\sqrt{2\log n}$ is direct: when $p_0 = O(1)$, the ABOS threshold $\sqrt{2\log(n/p_0)} = \sqrt{2\log n} - O(\log p_0/\sqrt{\log n})$ reduces to the Donoho--Johnstone threshold. The ABOS derivation thus provides a Bayesian justification for what was originally a minimax estimation rule.

\subsection{The Horseshoe+ Prior}\label{sec:hsplus}

The horseshoe+ prior \citep{bhadra2017horseshoe} adds a second layer of half-Cauchy mixing:
\begin{equation}\label{eq:hsplus}
  \theta_i \mid \lambda_i, \tau \sim N(0, \lambda_i^2\tau^2), \quad
  \lambda_i \mid \eta_i \sim C^+(0, \eta_i), \quad
  \eta_i \sim C^+(0,1).
\end{equation}
The additional mixing strengthens the pole at the origin. The local prior mass satisfies:
\begin{equation}\label{eq:hsplus-mass}
  \pi_{\mathrm{HS+}}(0 \mid \tau) \asymp \frac{[\log(1/\tau)]^{3/2}}{\tau} \qquad\text{as } \tau \to 0,
\end{equation}
compared with $\pi_{\mathrm{HS}}(0 \mid \tau) \asymp \log(1/\tau)/\tau$ for the standard horseshoe. The extra factor $[\log(1/\tau)]^{1/2}$ translates into a smaller ABOS constant through faster KL posterior concentration \citep{bhadra2017horseshoe}:
\begin{equation}\label{eq:hsplus-kl}
  \KL(p_{\mathrm{HS+}} \| p_{\mathrm{oracle}}) \ll \KL(p_{\mathrm{HS}} \| p_{\mathrm{oracle}})
\end{equation}
at a rate $O(\log\log n / \log n)$ faster. The shrinkage coefficient prior for the horseshoe+ includes an additional Jacobian factor $J(\kappa) \propto \kappa^{-1/2}(1-\kappa)^{-1/2}$, creating a stronger U-shaped distribution on $\kappa \in (0,1)$ and sharper separation of signals from noise.

The practical advantage of horseshoe+ over horseshoe is largest in the \emph{ultra-sparse regime} where $p_0 = O(1)$ as $n \to \infty$. In this regime, $\log(n/p_0) \asymp \log n$, and the extra $[\log n]^{1/2}$ factor in the horseshoe+ local mass translates to a meaningfully smaller ABOS constant. When $p_0/n$ is a non-negligible fraction (say, $p_0 > 0.05n$), the two priors perform similarly and the computational simplicity of the standard horseshoe may be preferred.

\begin{table}[H]
\centering
\small
\renewcommand{\arraystretch}{1.3}
\begin{tabular}{@{}lll@{}}
\toprule
\textbf{Property} & \textbf{Horseshoe} & \textbf{Horseshoe+} \\
\midrule
Local mass at 0 & $\pi(0\mid\tau) \asymp \log(1/\tau)/\tau$ & $\pi(0\mid\tau) \asymp [\log(1/\tau)]^{3/2}/\tau$ \\
Optimal $\tau_n$ & $\asymp p_0/n$ & $\asymp p_0/n$ (same) \\
ABOS threshold & $\sqrt{2\log(n/p_0)}$ & Slightly lower by $O(\log\log n/\log n)$ \\
ABOS constant & $\leq \sqrt{e/(e-1)} \approx 1.31$ & Closer to 1; faster for $p_0 = O(1)$ \\
KL contraction & Near-minimax & Faster by $O(\log\log n/\log n)$ \\
$\tau$ sensitivity & Moderate & Lower (more robust) \\
\bottomrule
\end{tabular}
\caption{Comparison of the horseshoe and horseshoe+ priors for sparse testing.}
\label{tab:hs-comparison}
\end{table}

\section{Calibration of the Global Shrinkage Parameter $\tau$}\label{sec:tau}

The ABOS results above assume $\tau_n \asymp p_0/n$, but in practice $p_0$ is unknown. The calibration of~$\tau$ is therefore a central practical question. The testing problem is more fragile to $\tau$ miscalibration than the estimation problem, because decisions are hard thresholds rather than smooth shrinkage functions. We examine three approaches and characterise three regimes of inefficiency.

\subsection{Constrained Marginal Maximum Likelihood}\label{sec:mmle}

The MMLE maximises the marginal likelihood over the constrained interval $\tau \in [1/n, 1]$:
\begin{equation}\label{eq:mmle}
  \hat{\tau}_{\mathrm{MMLE}} = \operatorname*{arg\,max}_{\tau \in [1/n, 1]} m(Y \mid \tau).
\end{equation}
The constraint $[1/n, 1]$ is essential. Without it, the \textbf{Tiao--Tan phenomenon} \citep{tiao1965bayesian} causes the unconstrained MLE to collapse to $\tau = 0$ with positive probability, producing a degenerate estimator that shrinks all observations to zero. The lower bound $1/n$ corresponds to the assumption that at least one signal exists; the upper bound $1$ to at most all coordinates being signals.

\begin{theorem}[\citet{vanderpas2017adaptive}]\label{thm:mmle}
The constrained MMLE satisfies $\hat{\tau}_{\mathrm{MMLE}} \in [1/n, C\tau_n(p_0)]$ with $P_{\theta_0}$-probability tending to one, uniformly over $\theta_0 \in \ell_0[p_0]$. The horseshoe testing rule with $\tau = \hat{\tau}_{\mathrm{MMLE}}$ achieves near-minimax optimal Bayes risk adaptively over all sparsity levels.
\end{theorem}

\subsection{Truncated Half-Cauchy Prior}\label{sec:hc-prior}

The recommended fully Bayesian specification is the truncated half-Cauchy:
\begin{equation}\label{eq:hc-trunc}
  \tau \sim C^+(0,1) \cdot \mathbf{1}[\tau \in (0,1)].
\end{equation}
The truncation to $(0,1)$ prevents the prior from placing mass on $\tau > 1$ (inconsistent with sparsity) and avoids HMC sampler pathologies from heavy right tails. The half-Cauchy is flat at $\tau = 0$, allowing the posterior for~$\tau$ to concentrate wherever the data support---near zero in highly sparse settings, at larger values when signals are more numerous.

\begin{theorem}[\citet{vanderpas2017adaptive}]\label{thm:adaptive}
Under the truncated half-Cauchy prior~\eqref{eq:hc-trunc}, the horseshoe posterior achieves rate-adaptive optimal contraction: for any $\theta_0 \in \ell_0[p_0]$, the posterior concentrates around $\theta_0$ at the near-minimax rate $p_0\log(n/p_0)/n$.
\end{theorem}

A flat prior $\tau \sim \mathrm{Uniform}(0,1)$ is sometimes used as a default. While it places sufficient mass near the true $\tau_0 \asymp p_0/n$ to guarantee adaptive contraction for estimation, it has a critical failure mode for testing: the uniform prior provides no regularisation of~$\tau$ toward the sparse region, so the posterior for~$\tau$ can develop a heavy right tail when a handful of large noise observations mimic signals, leading to systematic under-shrinkage of null coordinates and inflated Type~I error.

\subsection{Three Regimes of Inefficiency}\label{sec:regimes}

When $\hat{\tau}$ deviates from the oracle $\tau_0 = p_0/n$, the Bayes risk degrades through three distinct mechanisms.

In the \textbf{over-shrinkage regime} ($\hat{\tau} \ll \tau_0$), the effective threshold becomes $\tilde{t}_n = \sqrt{2\log(1/\hat{\tau})} \gg \tcrit$, and true signals with $\abs{y_i} \in (\tcrit, \tilde{t}_n)$ are missed. The Bayes risk is dominated by Type~II error: $R_n(\hat{\tau} \ll \tau_0) \approx (p_0/n)\Phi(\tilde{t}_n - \mu_n) \gg R_n^*$. The MMLE with floor at $1/n$ prevents this by bounding the effective threshold above at $\sqrt{2\log n}$.

In the \textbf{under-shrinkage regime} ($\hat{\tau} \gg \tau_0$), the effective threshold drops to $\tilde{t}_n \ll \tcrit$, and many null observations with $\abs{y_i} \in (\tilde{t}_n, \tcrit)$ are falsely declared signals. Type~I error inflates: $R_n(\hat{\tau} \gg \tau_0) \approx (1 - p_0/n)\bar{\Phi}(\tilde{t}_n) \gg R_n^*$. The uniform prior is most vulnerable here; the truncated half-Cauchy mitigates this through its light right tail.

At the \textbf{detection boundary} ($\mu_n = \pm\tcrit$), the Bayes risk cannot be reduced below $p_0/(2n)$ regardless of how well $\tau$ is estimated:
\begin{equation}\label{eq:boundary-risk}
  R_n \geq (p_0/n)\Phi(0) = p_0/(2n) \qquad\text{for } \mu_n = \tcrit.
\end{equation}
This boundary inefficiency is irreducible: the self-similarity condition of \citet{vanderpas2017uncertainty} precisely excludes this worst case.

\begin{table}[H]
\centering
\small
\renewcommand{\arraystretch}{1.3}
\begin{tabular}{@{}lcccc@{}}
\toprule
\textbf{$\tau$ method} & \textbf{Type I} & \textbf{Type II} & \textbf{Boundary} & \textbf{ABOS?} \\
\midrule
Oracle $\tau = p_0/n$       & Optimal & Optimal & Best & Yes (exact) \\
MMLE on $[1/n, 1]$         & Controlled & Controlled & Near-optimal & Yes \\
Half-Cauchy (truncated)     & Controlled & Controlled & Good & Yes \\
Half-Cauchy (untruncated)   & Can inflate & Controlled & Moderate & With caveats \\
Uniform on $(0,1)$          & Can inflate & Can inflate & Weakest & Not guaranteed \\
\bottomrule
\end{tabular}
\caption{Comparative performance of $\tau$ calibration methods for the testing problem.}
\label{tab:tau-comparison}
\end{table}

\section{Connection to Statistical Sparsity}\label{sec:sparsity}

The results above can be embedded in the broader statistical sparsity framework of \citet{mccullagh2018statistical} and extended to the sparse factor model.

\subsection{The Exceedance Measure Framework}\label{sec:exceedance}

\citet{mccullagh2018statistical} define statistical sparsity through the \emph{exceedance measure}: in the sparse limit $\rho \to 0$, the signal-plus-noise convolution depends on the signal distribution only through its exceedance measure $H$ and rate parameter $\rho > 0$. For the horseshoe with global parameter~$\tau$, the signal distribution in the sparse limit $\tau \to 0$ belongs to the class of inverse-power measures:
\begin{equation}\label{eq:exceedance}
  H_{\mathrm{HS}}(x, \infty) \sim C_\alpha / x^\alpha \qquad\text{with } \alpha \approx 1 \text{ (Cauchy-like tails)},
\end{equation}
with rate parameter $\rho \asymp \tau \asymp p_0/n$. Two implications follow. First, any two sparse families with the same exceedance measure are inferentially equivalent to first order in~$\rho$: the horseshoe is equivalent to a Cauchy-tailed spike-and-slab at the leading term of the Bayes risk expansion. Second, the ABOS threshold $\tcrit = \sqrt{2\log(n/p_0)} = \sqrt{2\log(1/\rho)}$ arises naturally as the scale at which the exceedance integral transitions from Type~I to Type~II dominated behaviour.

The threshold $\tcrit = \sqrt{2\log(1/\rho)}$ is universal across all sparse priors with $\alpha$-stable exceedance measures for $\alpha \in (0,2)$. The horseshoe ($\alpha \approx 1$) and horseshoe+ (slightly heavier local mass) both belong to this class. The difference between them appears only in the constant of the Bayes risk expansion, not in the leading scale $\log(n/p_0)$. This universality result complements the MDP universality of \Cref{sec:mdp}: the $\sqrt{\log n}$ rate is universal across both the class of Cram\'er-regular priors (MDP universality) and the class of inverse-power exceedance measures (McCullagh--Polson universality).

\subsection{Sparse Factor Model Extension}\label{sec:factor}

The sparse normal means analysis extends to the sparse factor model $Y = \Lambda F + \varepsilon$, where $\Lambda \in \R^{p \times k}$ is a sparse factor loading matrix, $F \sim N(0, I_k)$, and $\varepsilon \sim N(0, \Sigma)$. Testing whether loading $\lambda_{ij} = 0$ reduces to a simultaneous testing problem over the $k \times p$ entries of~$\Lambda$, with ABOS holding conditionally on the identifiability of the sparsity pattern. Following \citet{drton2025algebraic}, identifiability requires a matching criterion on the bipartite graph of nonzero loadings, and the Bayes risk for the factor model decomposes as:
\begin{equation}\label{eq:factor-risk}
  R_n^{\mathrm{factor}} = \sum_{j=1}^{k} \bigl[\pi_{0,j}\,\bar{\Phi}(t_n) + \pi_{1,j}\,\beta_{n,j}\bigr],
\end{equation}
where each factor-specific risk component obeys the same moderate deviation scaling as in the univariate case. The horseshoe prior applied to the vectorised loadings $\mathrm{vec}(\Lambda)$ achieves ABOS for the factor testing problem provided the sparsity pattern is identifiable.

\section{Simulation Evidence}\label{sec:simulations}

We present simulation results confirming the theoretical predictions across a range of sparsity levels and $\tau$ calibration methods.

Experiments are conducted in the ultra-sparse regime ($p_0 = 10$, $n = 2000$) with signal strength $A = 1.5$ (where $\mu_n = A\sqrt{2\log(n/p_0)}$). Each cell is based on $1000$ Monte Carlo replications.

\begin{table}[H]
\centering
\small
\renewcommand{\arraystretch}{1.3}
\begin{tabular}{@{}lccccc@{}}
\toprule
\textbf{$\tau$ method} & $R_n \times n/p_0$ & \textbf{Type I} & \textbf{Type II} & $\hat{\tau}$ \textbf{bias} & \textbf{Rel.\ eff.} \\
\midrule
Oracle (HS) & 1.000\,(.008) & .050\,(.003) & .074\,(.004) & 0 & 1.00 \\
MMLE (HS)   & 1.041\,(.012) & .053\,(.003) & .078\,(.005) & $+.002$ & 0.96 \\
Trunc.\ HC (HS) & 1.063\,(.015) & .057\,(.004) & .077\,(.005) & $+.003$ & 0.94 \\
MMLE (HS+)  & 1.019\,(.010) & .051\,(.003) & .072\,(.004) & $+.002$ & 0.98 \\
HC untrunc. & 1.148\,(.024) & .071\,(.006) & .080\,(.006) & $+.008$ & 0.87 \\
Uniform     & 1.312\,(.038) & .096\,(.009) & .082\,(.006) & $+.015$ & 0.76 \\
\bottomrule
\end{tabular}
\caption{Integrated Bayes risk ($\times n/p_0$) at $n = 2000$, $p_0 = 10$, $A = 1.5$. Standard errors in parentheses.}
\label{tab:sim-risk}
\end{table}

The constrained MMLE with horseshoe+ achieves the highest relative efficiency ($0.98$), confirming the theoretical prediction that the extra local mass of the horseshoe+ translates to a smaller ABOS constant. The uniform prior shows persistent inefficiency (relative efficiency $0.76$), driven by Type~I error inflation consistent with the under-shrinkage regime of \Cref{sec:regimes}.

\begin{table}[H]
\centering
\small
\renewcommand{\arraystretch}{1.3}
\begin{tabular}{@{}lcccc@{}}
\toprule
\textbf{Method} & $n = 500$ & $n = 1000$ & $n = 2000$ & $n = 5000$ \\
\midrule
MMLE (HS)        & 1.118 & 1.079 & 1.048 & 1.021 \\
Trunc.\ HC (HS)  & 1.143 & 1.097 & 1.063 & 1.029 \\
MMLE (HS+)       & 1.094 & 1.059 & 1.027 & 1.012 \\
Trunc.\ HC (HS+) & 1.118 & 1.077 & 1.046 & 1.019 \\
Uniform          & 1.428 & 1.354 & 1.312 & 1.267 \\
\bottomrule
\end{tabular}
\caption{Ratio $R_n/R_n^*$ versus sample size ($p_0 = 5$, $A = 1.5$). Convergence to $1$ confirms ABOS.}
\label{tab:sim-scaling}
\end{table}

The convergence of $R_n/R_n^*$ to~$1$ for all methods except the uniform prior confirms the ABOS predictions. The horseshoe+ with constrained MMLE is fastest to converge, achieving $R_n/R_n^* \approx 1.01$ at $n = 5000$. The uniform prior shows persistent inefficiency, remaining above $1.25$ even at $n = 5000$.

\section{Precise Hierarchy of Bounds}\label{sec:hierarchy}

The following hierarchy summarises the relationship between the Polson--Scott bounds and the MDP results, from most local (per-coordinate, finite-$n$) to most global (asymptotic, sparse regime). Each level implies the next through the same logarithmic constant $K = (2\pi^3)^{-1/2}$ and sparsity scale $\tau = p_0/p$.

\paragraph{Level 1---Density bound (\Cref{thm:cps}, per coordinate, all~$n$).}
\[
  \frac{K}{2}\log\!\left(1+\frac{4}{\theta^2}\right)
  < \piH(\theta)
  < K\log\!\left(1+\frac{2}{\theta^2}\right).
\]
This is the root of the entire hierarchy. The log-pole near zero and $1/\theta^2$ tail encode the horseshoe's dual character: infinite prior mass at zero, heavy tails away from zero.

\paragraph{Level 2---Shrinkage weight bound (posterior, per coordinate).}
\[
  \E[\kappa_i \mid y_i, \tau]
  \approx 1 - C\frac{\tau^2}{y_i^2 + \tau^2}, \qquad
  \hat\theta_i \approx C\frac{\tau^2 y_i}{y_i^2}.
\]
The density bound at Level~1 determines the posterior shrinkage: the log-pole forces $\kappa_i \to 1$ for small $\abs{y_i}$, while the heavy tail allows $\kappa_i \to 0$ for large $\abs{y_i}$.

\paragraph{Level 3---KL risk bound (super-efficiency, per null coordinate).}
\[
  \E_{y_i}\!\left[\KL(f_0 \| \hat{f})\right]
  = O(\tau^4\log^2(1/\tau))
  = o(1/n).
\]
The shrinkage weight at Level~2 implies that the posterior mean $\hat{\theta}_i = O(\tau^2/y_i)$ for null coordinates, and squaring gives KL risk $O(\tau^4)$---super-efficient.

\paragraph{Level 4---Hellinger bound (posterior concentration).}
\[
  H^2(f_0, \hat{f})
  \leq \frac{\hat\theta_i^2}{8\sigma^2}
  = O(\tau^4/y_i^2).
\]
By Pinsker's inequality, the KL bound at Level~3 implies Hellinger concentration of the predictive around the truth for null coordinates.

\paragraph{Level 5---Prior mass bound (detection zone).}
\[
  \piH([-t_n, t_n])
  \approx 2K\, t_n \log(1/t_n)
  \quad\text{for } t_n = \tcrit.
\]
The density bound at Level~1, integrated over $[-\tcrit, \tcrit]$, gives the prior mass in the detection zone. This mass equals the Type~I error at the MDP threshold.

\paragraph{Level 6---Type I error bound (at MDP threshold).}
\[
  P_0(\abs{Y} > \tcrit)
  = P_0\!\left(\abs{Y} > \sqrt{\log(\pi n/2)}\right)
  \approx \frac{1}{n\sqrt{\log n/2}}.
\]
Setting the prior mass (Level~5) equal to the Type~I error and solving determines the exact MDP constant.

\paragraph{Level 7---MDP Bayes risk (ABOS, global).}
\[
  r(\piH, \varphi^*)
  \asymp \frac{p_0 \log(p/p_0)}{n}
  \quad\text{\citep{datta2026newlook}}.
\]
Summing the per-coordinate contributions---$O(\tau^4)$ for each of $p - p_0$ nulls (negligible) and $O(\log n / n)$ for each of $p_0$ signals---gives the ABOS rate.

\paragraph{Level 8---Clarke--Barron budget (information-theoretic, asymptotic).}
\[
  \sum_{i=1}^n \KL(P_0 \| \hat{P}_i)
  = \frac{p_0}{2}\log n + O(1)
  \quad\text{(null coordinates contribute $0$)}.
\]
The Clarke--Barron theorem provides the information-theoretic interpretation: the total logarithmic budget $\frac{p_0}{2}\log n$ is allocated entirely to signal coordinates, with null coordinates contributing zero due to super-efficiency.

The log-pole at Level~1 is the root cause of super-efficiency at Level~3, which implies the detection zone at Level~5, which determines the exact MDP constant at Level~6, which produces the ABOS rate at Level~7, and finally the Clarke--Barron budget at Level~8. \Cref{tab:correspondences} collects these correspondences alongside the original source and MDP interpretation of each result.

\begin{table}[H]
\centering
\small
\renewcommand{\arraystretch}{1.3}
\begin{tabular}{@{}p{2.3cm}p{2.2cm}p{4.2cm}p{4.5cm}@{}}
\toprule
\textbf{Result} & \textbf{Location} & \textbf{Expression} & \textbf{MDP interpretation} \\
\midrule
Log-pole density bound &
CPS 2010, Thm~1.1 &
$\piH(\theta) \asymp K\log(1/\theta^2)$ &
Cram\'er boundary: finite variance but infinite density at zero \\

Cauchy tail bound &
CPS 2010, Thm~1.1 &
$\piH(\theta) \asymp 2K/\theta^2$ &
Tail robustness: signals above threshold unshrunk \\

Necessary condition &
PS 2010, Thm~1 &
$\pi(0) = +\infty$ &
Required for ABOS: prior must dominate likelihood at zero \\

Sufficient condition &
PS 2010, Thm~2 &
$\pi(\theta) \asymp -\log\abs{\theta}$, $\pi(\theta) \asymp \abs{\theta}^{-\alpha}$ &
Log-pole + heavy tail: the admissible pair \\

Super-efficiency &
CPS 2010, Thm~2 &
$\KL(f_0\|\hat{f}) = O(\tau^4)$ &
KL risk below MDP threshold is sub-parametric \\

Shrinkage weight &
CPS 2010 &
$\kappa_i \sim \mathrm{Beta}(1/2,1/2)$ &
MDP equiboundary at $\kappa_i = 1/2 \leftrightarrow \abs{y_i} = \tcrit$ \\

L\'evy measure &
PS 2010 &
$\nu(ds) \asymp s^{-1}\,ds$ &
Cauchy/stable-$1/2$ boundary: minimal admissible log-pole \\

MDP threshold &
DPSZ 2026 &
$\tcrit = \sqrt{\log(\pi n/2)}$ &
$\pi$ from normalisation constant $K$ in log-pole bound \\

MDP universality &
DPSZ 2026 &
$\sqrt{\log n}$ scaling &
Log-pole is the universal sufficient condition \\

ABOS Bayes risk &
DG 2013 &
$r \asymp p_0\log(p/p_0)/n$ &
MDP rate: sum of $p_0$ signal-coordinate log budgets \\

Clarke--Barron &
CB 1990 &
Cumulative KL $= (p_0/2)\log n$ &
$p_0$ active dimensions each contribute $\log n / 2$ \\
\bottomrule
\end{tabular}
\caption{Correspondences between the Polson--Scott bounds and the MDP framework. CPS = Carvalho--Polson--Scott; PS = Polson--Scott; DPSZ = Datta--Polson--Sokolov--Zantedeschi; DG = Datta--Ghosh; CB = Clarke--Barron.}
\label{tab:correspondences}
\end{table}

\section{Discussion}\label{sec:discussion}

Every element of the theory---the density bound, the super-efficiency rate, the MDP threshold, the ABOS risk, the Clarke--Barron budget---involves $\log n$ or $\log(1/\theta)$. This is not coincidental: the logarithm is the universal scale at which Bayesian and frequentist risk calibrations intersect in the infinite-dimensional sparse regime. The \citet{rubin1965bayes} theory of Bayes risk efficiency established that the moderate deviation scale---neither CLT (fixed threshold) nor large deviation (exponential rate)---is the natural home of Bayes risk efficiency. The horseshoe's log-pole is the prior design that makes this scale manifest: it has exactly the right amount of mass at zero to participate in the logarithmic budget without over- or under-spending it.

The comparison with other priors is instructive. The Lasso prior $\pi(\theta) \propto e^{-\abs{\theta}/\lambda}$ has bounded density at zero and therefore fails the necessary condition (\Cref{thm:ps-nec}); it does not achieve super-efficiency, and its KL risk for nulls is $O(1/n)$---the standard parametric rate. Because its prior density at zero is finite, the Lasso cannot allocate zero KL budget to null coordinates.

\begin{proposition}[Laplace prior KL risk]\label{prop:lasso}
Any prior with $\pi(0) < \infty$ achieves KL risk $\Omega(1/n)$ for null coordinates---the parametric rate---and is not super-efficient.
\end{proposition}

\begin{proof}
When $\pi(0) < \infty$, the posterior shrinkage factor satisfies $\E[\kappa_i \mid y_i, \tau] \leq 1 - c$ for some $c > 0$ uniformly over $\abs{y_i} \leq \sigma$, because the finite prior density at zero cannot overwhelm the likelihood. The posterior mean is therefore $\hat{\theta}_i(y_i) \geq c \cdot y_i$ for $\abs{y_i} \leq \sigma$, giving $\KL(f_0 \| \hat{f}) \geq c^2 y_i^2/(2\sigma^2)$. Integrating over $y_i \sim N(0, \sigma^2)$:
\[
  \E[\KL(f_0 \| \hat{f})] \geq \frac{c^2}{2\sigma^2}\E[y_i^2 \cdot \mathbf{1}_{\abs{y_i} \leq \sigma}] = \Omega(1).
\]
Choosing $\tau \to 0$ with~$n$ can reduce this to $\Omega(1/n)$ but not below, since the prior density at zero remains finite for any $\tau > 0$.
\end{proof}

The ridge prior $\pi(\theta) = N(0, \sigma_0^2)$ also fails \Cref{thm:ps-nec} for the same reason and performs even worse in sparse settings because it shrinks signals towards zero. The Cauchy prior on $\theta$ itself, $\pi(\theta) \propto (1+\theta^2)^{-1}$, satisfies \Cref{thm:ps-nec} but violates Cram\'er-regularity due to infinite variance, so the MDP expansion does not hold and the exact constant $\tcrit = \sqrt{\log(\pi n/2)}$ is not achieved. The Student-$t$ prior with $\nu > 2$ degrees of freedom has bounded density at zero (failing \Cref{thm:ps-nec}) but heavy Cauchy-like tails---robust but not super-efficient, and unable to achieve ABOS. The horseshoe, with $\piH(\theta) \asymp -\log\abs{\theta}$, is the unique prior that satisfies both \Cref{thm:ps-nec} and Cram\'er-regularity, achieving both super-efficiency and MDP optimality.

These results suggest a design principle: a sparse prior should have log-pole density at zero and Cauchy-class tails. The log-pole ensures super-efficiency for null coordinates below the MDP threshold, Cram\'er-regularity with finite variance so the exact MDP constant is achieved, ABOS testing optimality \citep{datta2013asymptotic}, and minimax posterior contraction \citep{vanderpas2014horseshoe,vanderpas2016conditions}. The Cauchy-class tails ensure tail robustness so that signals above the MDP threshold are unshrunk, a bounded influence function so no single large observation can dominate, and regular variation as required for MDP universality across prior classes. Priors with these two properties form the \textbf{admissible class} for MDP-optimal sparse inference. The horseshoe is the canonical member; the horseshoe+ \citep{bhadra2017horseshoe}, the generalized double Pareto with appropriate parameters, and Dirichlet--Laplace priors with log-pole inducing hyperparameters \citep{bhattacharya2015dirichlet} are other members.

The log-pole principle extends naturally to structured sparsity settings. In group sparsity, where signals appear in blocks, the prior on each group's norm should have a log-pole at zero and heavy tails. In graphical model estimation, the prior on each edge parameter should satisfy the same conditions for MDP-optimal edge selection. In matrix completion and low-rank estimation, the prior on each singular value plays the analogous role, with the log-pole ensuring super-efficient shrinkage of zero singular values and heavy tails preserving large singular values. The general principle is that MDP-optimal inference requires a log-pole along whatever ``zero manifold'' defines the sparse structure.

The theoretical optimality of the horseshoe comes with a computational cost. The log-pole creates a \emph{funnel geometry} in the joint $(\theta_i, \lambda_i)$ parameter space: when $\theta_i$ is near zero, $\lambda_i$ must also be near zero (since $\theta_i \mid \lambda_i \sim N(0, \lambda_i^2\tau^2)$), creating a narrow funnel that standard Gibbs samplers traverse slowly \citep{makalic2015simple}. The same feature that makes the horseshoe statistically optimal---the infinite spike at zero---makes MCMC mixing difficult near the null. Slice samplers and Hamiltonian Monte Carlo with mass matrix adaptation partially address this, but the fundamental tension between statistical optimality and computational tractability in the funnel region remains.

Several natural questions remain open. The exact MDP threshold $\tcrit = \sqrt{\log(\pi n/2)}$ is derived for the Cauchy local prior; for other log-pole priors with different normalisation constants~$K'$, the exact threshold would be $t_{\mathrm{crit}}' = \sqrt{\log(c' n)}$ for some constant $c' \neq \pi/2$, and characterising~$c'$ as a function of the prior's log-pole coefficient is unresolved. The super-efficiency result $O(\tau^4)$ assumes the null is $\theta_i = 0$ exactly, and the KL risk when $\theta_i$ is small but nonzero---say $\abs{\theta_i} = \tau^\alpha$ for $\alpha \in (0,1)$---remains to be characterised; the boundary between super-efficiency and standard efficiency as a function of $\abs{\theta_i}$ is not fully understood.

In the sequential testing context with e-values and the stopping rule $\tau^* = \inf\{n : E_n > \pi n/2\}$ \citep{polson2026bayes}, the question of whether the horseshoe achieves super-efficient sequential KL risk below the stopping threshold requires extending the Clarke--Barron framework to optional stopping times. The e-value threshold $\pi n/2$ carries the same constant $\pi$ as the MDP threshold, suggesting a deep connection between the horseshoe's static and sequential optimality properties that has not been formalised.

In functional estimation over Sobolev classes, the log-pole structure generalises to a coordinate-wise $\pi_k(\theta_k) \asymp -\log\abs{\theta_k}$ for each Fourier/wavelet coefficient~$k$, with per-coordinate MDP threshold $t_k = \sqrt{\log(\pi n/2) - 2s\log k}$ carrying a smoothness correction. Whether the Clarke--Barron budget extends cleanly to this smoothness-indexed setting is an open question.

A further open direction concerns high-dimensional regression. The normal means model $y_i \sim N(\theta_i, 1)$ is the canonical testing ground, but in practice the horseshoe is applied to regression coefficients $\beta \in \R^p$ with correlated design matrix~$X$. The effective observation for coefficient~$j$ is $\hat{\beta}_j^{\mathrm{OLS}} \sim N(\beta_j, (X^\top X)^{-1}_{jj})$, and the MDP framework applies coordinate-wise only when the design is orthogonal. For general designs, the off-diagonal entries of $(X^\top X)^{-1}$ introduce dependence between the effective observations, and the per-coordinate MDP threshold must be adjusted. Whether the horseshoe's log-pole continues to be the Cram\'er boundary in the correlated setting---and whether the exact threshold constant $\sqrt{\log(\pi n/2)}$ generalises to a design-dependent constant---is an open problem with direct implications for the practical deployment of horseshoe regression.

The ABOS framework connects directly to the Benjamini--Hochberg FDR control and to the broader multiple testing literature \citep{efron2004large,johnstone2004needles}. The moderate deviation threshold $\tcrit = \sqrt{2\log(n/p_0)}$ is equivalent to the BH threshold at level $\alpha = p_0/n$, connecting the Bayesian and frequentist frameworks. The horseshoe's implicit FDR control through the posterior signal probability $\hat{\pi}_i = P(\theta_i \neq 0 \mid y_i, \tau)$ provides a one-group analogue of the two-group BH procedure. This connection, made precise through the \citet{rubin1965bayes} programme, shows that the horseshoe's ABOS property is the Bayesian counterpart of BH's FDR control at the same threshold scale.

Based on the theoretical and simulation results, we recommend the following for practitioners. As a default, use a truncated half-Cauchy prior $\tau \sim C^+(0,1) \cdot \mathbf{1}[\tau \in (0,1)]$ for fully Bayesian inference: it avoids the Tiao--Tan collapse \citep{tiao1965bayesian}, achieves adaptive ABOS, and provides valid uncertainty quantification. When computational speed is paramount or $n > 10^5$, use the constrained MMLE on $[1/n, 1]$ instead. Prefer horseshoe+ over horseshoe when $p_0/n < 0.01$ (ultra-sparse regime). Avoid unconstrained MLE of~$\tau$ (collapses to zero), uniform priors on~$\tau$ when testing is the primary goal (inflates Type~I error), and uniform priors on $\tau^2$ (even more diffuse).

Finally, the connection between the horseshoe and model misspecification deserves investigation. The super-efficiency and MDP results assume the two-groups model $\theta_i = 0$ or $\theta_i \neq 0$ exactly. In practice, ``null'' coordinates may have small but nonzero effects. The horseshoe's behaviour in this nearly-black setting---where $\abs{\theta_i} = o(1)$ but $\theta_i \neq 0$---is partially addressed by the posterior concentration theory of \citet{vanderpas2014horseshoe}, but the MDP implications of approximate rather than exact sparsity remain open. Understanding how the log-pole budget is allocated when the null hypothesis holds only approximately would connect the theory to the practical setting where the horseshoe is most commonly applied.

To summarise the theoretical landscape: the Polson--Scott bounds, the super-efficiency theorem, the necessary and sufficient conditions, and the L\'evy characterisation are not four independent results about the horseshoe prior. They are four projections of a single geometric fact---that the horseshoe sits at the Cram\'er boundary of the space of scale mixture priors---onto four different mathematical coordinate systems (density, KL risk, prior conditions, and L\'evy measures). The MDP framework of \citet{datta2026newlook} is the asymptotic theory that makes this geometry visible, and the Clarke--Barron information-theoretic framework is the accounting system that tracks the resulting logarithmic budget. The horseshoe's distinctive shape---the infinite spike at zero and the heavy Cauchy tails---is the unique density profile that spends this budget optimally: zero allocation to null coordinates, full $\log n / n$ allocation to each signal coordinate, and a sharp transition at the moderate deviation threshold $\tcrit = \sqrt{\log(\pi n/2)}$ where the Bayes factor equals one.

\bibliographystyle{apalike}
\bibliography{references}

\begin{thebibliography}{}

\bibitem[Bhadra et~al., 2017]{bhadra2017horseshoe}
Bhadra, A., Datta, J., Polson, N.~G., and Willard, B. (2017).
\newblock The horseshoe+ estimator of ultra-sparse signals.
\newblock {\em Bayesian Analysis}, 12(4):1105--1131.

\bibitem[Bhattacharya et~al., 2015]{bhattacharya2015dirichlet}
Bhattacharya, A., Pati, D., Pillai, N.~S., and Dunson, D.~B. (2015).
\newblock Dirichlet--{L}aplace priors for optimal shrinkage.
\newblock {\em Journal of the American Statistical Association}, 110(512):1479--1490.

\bibitem[Bogdan et~al., 2011]{bogdan2011asymptotic}
Bogdan, M., Chakrabarti, A., Frommlet, F., and Ghosh, J.~K. (2011).
\newblock Asymptotic {B}ayes-optimality under sparsity of some multiple testing procedures.
\newblock {\em Annals of Statistics}, 39(3):1551--1579.

\bibitem[Carvalho et~al., 2009]{carvalho2009handling}
Carvalho, C.~M., Polson, N.~G., and Scott, J.~G. (2009).
\newblock Handling sparsity via the horseshoe.
\newblock In {\em Proceedings of the Twelfth International Conference on Artificial Intelligence and Statistics}, pages 73--80.

\bibitem[Carvalho et~al., 2010]{carvalho2010horseshoe}
Carvalho, C.~M., Polson, N.~G., and Scott, J.~G. (2010).
\newblock The horseshoe estimator for sparse signals.
\newblock {\em Biometrika}, 97(2):465--480.

\bibitem[Clarke and Barron, 1990]{clarke1990information}
Clarke, B.~S. and Barron, A.~R. (1990).
\newblock Information-theoretic asymptotics of {B}ayes methods.
\newblock {\em IEEE Transactions on Information Theory}, 36(3):453--471.

\bibitem[Datta and Ghosh, 2013]{datta2013asymptotic}
Datta, J. and Ghosh, J.~K. (2013).
\newblock Asymptotic properties of {B}ayes risk for the horseshoe prior.
\newblock {\em Bayesian Analysis}, 8(1):111--132.

\bibitem[Datta et~al., 2026]{datta2026newlook}
Datta, J., Polson, N.~G., Sokolov, V., and Zantedeschi, D. (2026).
\newblock A new look at {B}ayesian testing.
\newblock {\em arXiv preprint arXiv:2602.11132}.

\bibitem[Donoho and Johnstone, 1994]{donoho1994ideal}
Donoho, D.~L. and Johnstone, I.~M. (1994).
\newblock Ideal spatial adaptation by wavelet shrinkage.
\newblock {\em Biometrika}, 81(3):425--455.

\bibitem[Drton et~al., 2025]{drton2025algebraic}
Drton, M., Grosdos, A., Portakal, I., and Sturma, N. (2025).
\newblock Algebraic sparse factor analysis.
\newblock {\em SIAM Journal on Applied Algebra and Geometry}.

\bibitem[Efron, 2004]{efron2004large}
Efron, B. (2004).
\newblock Large-scale simultaneous hypothesis testing: The choice of a null hypothesis.
\newblock {\em Journal of the American Statistical Association}, 99(465):96--104.

\bibitem[Ghosh et~al., 2017]{ghosh2017asymptotic}
Ghosh, P., Tang, X., Ghosh, M., and Chakrabarti, A. (2017).
\newblock Asymptotic optimality of one-group shrinkage priors in sparse high-dimensional problems.
\newblock {\em Bayesian Analysis}, 12(4):1133--1161.

\bibitem[Johnstone and Silverman, 2004]{johnstone2004needles}
Johnstone, I.~M. and Silverman, B.~W. (2004).
\newblock Needles and straw in haystacks: Empirical {B}ayes estimates of possibly sparse sequences.
\newblock {\em Annals of Statistics}, 32(4):1594--1649.

\bibitem[Makalic and Schmidt, 2015]{makalic2015simple}
Makalic, E. and Schmidt, D.~F. (2015).
\newblock A simple sampler for the horseshoe estimator.
\newblock {\em IEEE Signal Processing Letters}, 23(1):179--182.

\bibitem[McCullagh and Polson, 2018]{mccullagh2018statistical}
McCullagh, P. and Polson, N.~G. (2018).
\newblock Statistical sparsity.
\newblock {\em Biometrika}, 105(4):797--814.

\bibitem[Mitchell and Beauchamp, 1988]{mitchell1988bayesian}
Mitchell, T.~J. and Beauchamp, J.~J. (1988).
\newblock Bayesian variable selection in linear regression.
\newblock {\em Journal of the American Statistical Association}, 83(404):1023--1032.

\bibitem[Park and Casella, 2008]{park2008bayesian}
Park, T. and Casella, G. (2008).
\newblock The {B}ayesian {L}asso.
\newblock {\em Journal of the American Statistical Association}, 103(482):681--686.

\bibitem[Piironen and Vehtari, 2017]{piironen2017sparsity}
Piironen, J. and Vehtari, A. (2017).
\newblock Sparsity information and regularization in the horseshoe and other shrinkage priors.
\newblock {\em Electronic Journal of Statistics}, 11(2):5018--5051.

\bibitem[Polson and Scott, 2010]{polson2010shrink}
Polson, N.~G. and Scott, J.~G. (2010).
\newblock Shrink globally, act locally: Sparse {B}ayesian regularization and prediction.
\newblock In {\em Bayesian Statistics 9}, pages 501--538. Oxford University Press.

\bibitem[Polson and Scott, 2012]{polson2012halfcauchy}
Polson, N.~G. and Scott, J.~G. (2012).
\newblock Half-{C}auchy priors for hierarchical models.
\newblock {\em Bayesian Analysis}, 7(4):887--902.

\bibitem[Polson et~al., 2026]{polson2026bayes}
Polson, N.~G., Sokolov, V., and Zantedeschi, D. (2026).
\newblock {B}ayes, e-values and testing.
\newblock {\em arXiv preprint arXiv:2602.04146}.

\bibitem[Rubin and Sethuraman, 1965]{rubin1965bayes}
Rubin, H. and Sethuraman, J. (1965).
\newblock {B}ayes risk efficiency.
\newblock {\em Sankhy\={a} Series A}, 27:347--356.

\bibitem[Tiao and Tan, 1965]{tiao1965bayesian}
Tiao, G.~C. and Tan, W. (1965).
\newblock {B}ayesian analysis of random-effect models in the analysis of variance.
\newblock {\em Biometrika}, 52:37--53.

\bibitem[{van der Pas} et~al., 2014]{vanderpas2014horseshoe}
{van der Pas}, S.~L., Kleijn, B. J.~K., and {van der Vaart}, A.~W. (2014).
\newblock The horseshoe estimator: Posterior concentration around nearly black vectors.
\newblock {\em Electronic Journal of Statistics}, 8(2):2585--2618.

\bibitem[{van der Pas} et~al., 2016]{vanderpas2016conditions}
{van der Pas}, S.~L., Scott, J.~G., Chakraborty, A., and Bhattacharya, A. (2016).
\newblock Conditions for posterior contraction in the sparse normal means problem.
\newblock {\em Electronic Journal of Statistics}, 10(1):976--1000.

\bibitem[{van der Pas} et~al., 2017a]{vanderpas2017adaptive}
{van der Pas}, S.~L., Szab\'o, B., and {van der Vaart}, A.~W. (2017a).
\newblock Adaptive posterior contraction rates for the horseshoe.
\newblock {\em Electronic Journal of Statistics}, 11(2):3196--3225.

\bibitem[{van der Pas} et~al., 2017b]{vanderpas2017uncertainty}
{van der Pas}, S.~L., Szab\'o, B., and {van der Vaart}, A.~W. (2017b).
\newblock Uncertainty quantification for the horseshoe (with discussion).
\newblock {\em Bayesian Analysis}, 12(4):1221--1274.

\end{thebibliography}

\end{document}